\documentclass[12pt,draftcls,onecolumn,english]{IEEEtran}
\IEEEoverridecommandlockouts
\usepackage{cite}
\usepackage{ifthen}
\usepackage{graphicx}
\usepackage{verbatim}
\usepackage{subfig}
\usepackage{multirow}
\usepackage[usenames,dvipsnames]{color}
\usepackage{amssymb,amsmath}
\usepackage{xcolor}
\usepackage{multirow}
\usepackage{url}
\usepackage{multicol}
\usepackage{amsmath}
\usepackage{dcolumn}
\usepackage{amsfonts}
\usepackage{latexsym}
\usepackage{graphicx,epstopdf}
\usepackage{color}
\usepackage{amsthm}
\usepackage{subeqnarray}
\usepackage{authblk}

\allowdisplaybreaks

\newboolean{showcomments}
\setboolean{showcomments}{true}
\newcommand{\lina}[1]{  \ifthenelse{\boolean{showcomments}}
{ \textcolor{red}{(Lina says:  #1)}} {}  }
\newcommand{\Guannan}[1]{  \ifthenelse{\boolean{showcomments}}
{ \textcolor{blue}{(Guannan says:  #1)}} {}  }

\newcommand{\lijun}[1]{\ifthenelse{\boolean{showcomments}}
{ \textcolor{blue}{(Munzer says: #1)} } {} }

\def\ba{\begin{array}}
\def\ea{\end{array}}
\newcommand{\beq}{\begin{equation}}
\newcommand{\eeq}{\end{equation}}
\newcommand{\bq}{\begin{eqnarray}}
\newcommand{\eq}{\end{eqnarray}}
\newcommand{\bqn}{\begin{eqnarray*}}
\newcommand{\eqn}{\end{eqnarray*}}
\newcommand{\bee}{\begin{enumerate}}
\newcommand{\eee}{\end{enumerate}}
\newcommand{\bi}{\begin{itemize}}
\newcommand{\ei}{\end{itemize}}

\newcommand{\btab}{\begin{tabular}}
\newcommand{\etab}{\end{tabular}}

\newtheorem{theorem}{Theorem}

\newtheorem{proposition}[theorem]{Proposition}
\newtheorem{lemma}[theorem]{Lemma}

\newtheorem{assumption}{Assumption}

\makeatother

\usepackage{babel}

\begin{document}

\title{Real-time Decentralized and Robust Voltage Control in Distribution Networks}
\author{ Guannan Qu, Na Li, and Munther Dahleh   \thanks{Some preliminary results have been published at the 2014 Allerton Conference on Communication, Control and Computing.   This journal version contains several new results that are not contained in the conference paper, especially the robustness of the control algorithms and the related proofs. }\thanks{G. Qu and N. Li are with the School of Engineering and Applied Sciences, Harvard University, Cambridge, MA 02138, USA (Emails: gqu@g.harvard.edu, nali@seas.harvard.edu). M. Dahleh are with the Laboratory of Information and Decision Systems, Massachusetts Institute of Technology, Cambridge, MA 02139, USA (Email: dahleh@mit.edu.) \textit{Corresponding Author: N. Li}}}

\maketitle

\thispagestyle{plain}
\pagestyle{plain}

\begin{abstract}
Voltage control plays an important role in the operation of electricity distribution networks, especially when there is a large penetration of renewable energy resources. In this paper, we focus on voltage control through reactive power compensation and study how different information structures affect the control performance. In particular, we first show that using only voltage measurements to determine reactive power compensation is insufficient to maintain voltage in the acceptable range. Then we propose two fully decentralized and robust algorithms by adding additional information, which can stabilize the voltage in the acceptable range. The one with higher complexity can further minimize a cost of reactive power compensation in a particular form. Both of the two algorithms use only local measurements and local variables and require no communication. In addition, the two algorithms are robust against heterogeneous update rates and delays. 
\end{abstract}

\section{Introduction}
Voltages in a distribution feeder fluctuate according to the feeder loading condition. The primary purpose of voltage control is to maintain acceptable voltages (plus or minus 5\% around nominal values) at all buses along the distribution feeder under all possible operating conditions. Traditionally the voltage control is achieved by re-configuring transformer taps and capacitors banks (Volt/Var control) \cite{Baran1989a, Baran1989b} based on local measurements (usually voltages) at a slow time scale. This control setting works under normal circumstances, because the change of the loading condition is relatively mild and predictable.

Due to the increasing penetration of distributed energy resources (DER) such as photovoltaic and wind generation in the distribution networks, the operating conditions (supply, demand, voltages, etc) of the distribution feeder fluctuate fast and by a large amount. The conventional voltage control lacks flexibility to respond to those conditions and they may not produce the desired results. This raises important issues on the network security and reliability. To overcome the challenges, new technologies are being proposed and developed, e.g., the inverter design for voltage control. Inverters connect DERs to the grid and adjust the reactive power outputs to stabilize the voltages at a fast time scale \cite{smith2011smart,turitsyn2011options}. The new technologies will enable realtime
distributed voltage control that is needed for the future power grid.

One key element to implement those new technologies is the voltage control algorithms which satisfy certain information constraints yet guarantee the overall system performance.  In general, in the low/medium voltage distribution networks, only a small portion of buses are monitored, individuals are unlikely to announce their generation or load profile, the availability of DERs are fluctuating and uncertain, and even the grid parameters and the topology are only partially known. All of these facts make decentralized algorithms necessary for the voltage control. Each control component adjusts its reactive power input based on the local signals that are easy to measure, calculate, and/or communicate. Such local information dependence facilitates the realtime implementation of the algorithms. In fact, there exist classes of inverter-based local voltage control schemes that only use local voltage measurements \cite{aalborg2012droop,ali2012microgrid}.  For example, voltage droop control in microgrids adjust reactive power injection based on the local voltage deviation from its nominal value. However, it remains as a daunting challenge to guarantee the performance of the control rules, i.e., to  stabilize the voltages within the acceptable range under all possible operating conditions.  

\begin{figure}[ht]
\centering
\includegraphics[width=3.3in]{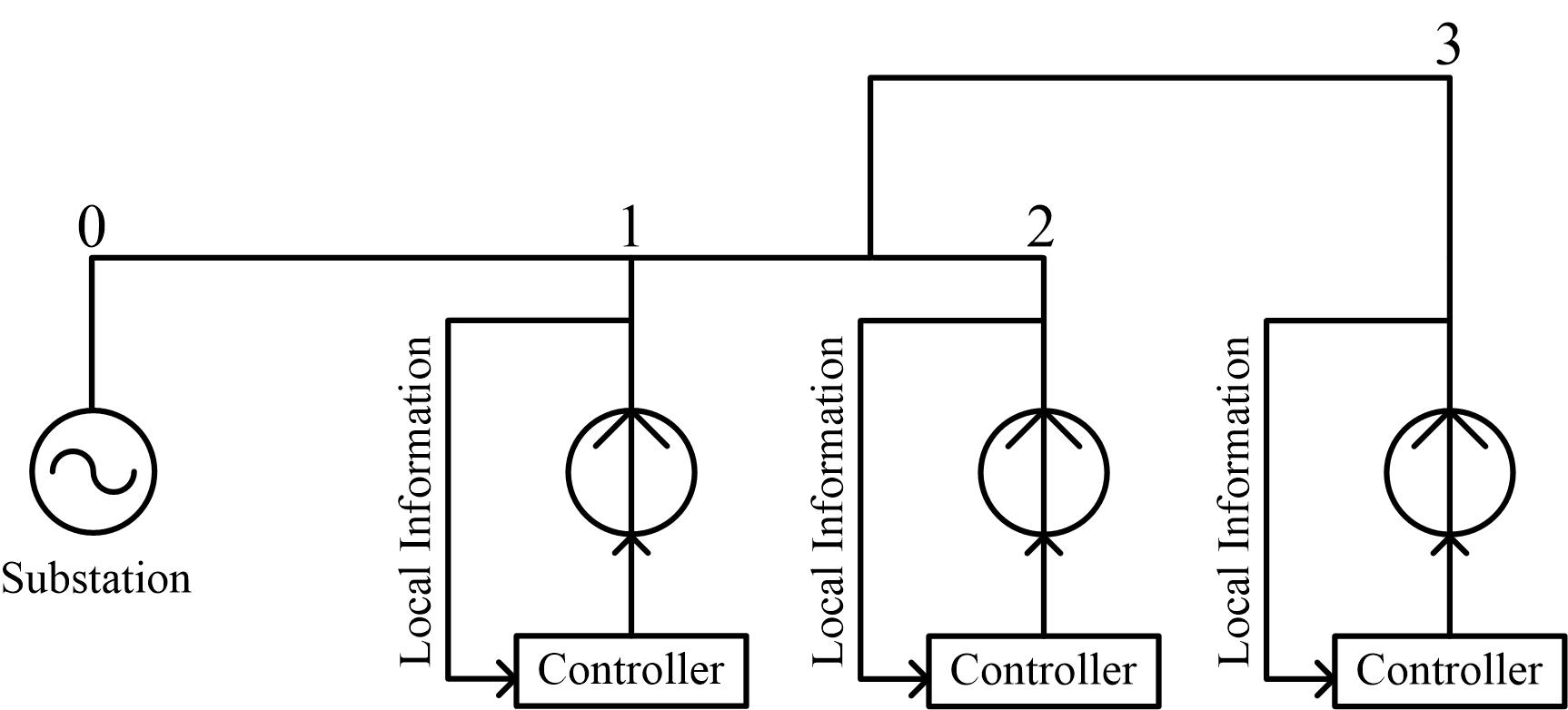}
\caption{An example of a distribution network with 1 substation and 3 buses. Each bus is equipped with a controller that only takes local information as its input.}
\label{fig:substation}
\end{figure}
In this paper, we focus on voltage control through reactive power compensation.\footnote{Thus, the two terms, voltage control and Volt/Var control, are interchangeable in this paper.} To facilitate the design and analysis of voltage control, we use a linear branch
flow model similar to the Simplified DistFlow equations
introduced in \cite{Baran1989c}. The linear branch flow model and the
local Volt/Var control form a closed loop dynamical system
(Equation (\ref{eq:dynamicsystem})). Then we study three types of decentralized voltage control using different local information. The scheme of the decentralized voltage control is shown in Figure ~\ref{fig:substation}. Note that there is no communication in the scheme.
In particular, we first show that using only voltage measurements to determine reactive power compensation (no matter whether in a centralized or decentralized way) is insufficient, because there always exist operating conditions under which such voltage control fails to maintain acceptable voltages. This implies that most of existing voltage control algorithms such as \cite{aalborg2012droop,IEEE1547} that only uses voltage measurements are not able to maintain voltages within the acceptable range under certain circumstances.  Then we propose two decentralized algorithms by adding additional information into the control design.\footnote{For the sake of clarity, we will refer the three algorithms as algorithm 0, 1, 2 correspondingly.} The additional information can be either measured locally or computed locally, meaning that the two algorithms are fully decentralized, requiring no communication. With the aid of the additional information, both of the two algorithms can stabilize voltages in  the acceptable range; and the one with higher complexity can further minimize a cost of the reactive power compensation in a particular form.  Moreover, the two algorithms are robust against heterogeneous update rates and delays. The heterogeneous update rates means that each control device can have its own time clock to choose when to update their reactive power injection and the clock rates can be different. The delays include both the measurement and the actuation delays that widely exist in real systems. The robustness makes the two algorithms suitable for practical applications. Overall, our results imply that with the aid of right local information, local voltages actually carry out the whole network information for Volt/Var control. 

\subsection{Related Work}
Research has been proposed and conducted to improve the existing voltage control to mitigate the voltage fluctuation impact. To name a few, \cite{carvalho2013distributed} studies the active power curtailment to mitigate the voltage rise impact caused by DER; \cite{bolognani2013distributed} studies the distributed VAR control to minimize power losses and stabilize voltages; \cite{farivar2013equilibrium} reverse-engineers the IEEE 1547.8 standard and study the equilibrium and dynamics of voltage control; \cite{robbins2013two} proposes two stage voltage control. Compared with the work in this paper, they usually require certain amount of communication or lack theoretical guarantee of performance regrading stabilizing the voltage within the acceptable range. A recent paper \cite{zhang2013local} proposes a similar algorithm as our algorithm 1. Besides the different approximation model used in \cite{zhang2013local}, the control objective is also different. \footnote{The objective of the voltage control in \cite{zhang2013local} is to drive the voltage profile to a predefined point, whereas the objective in this paper is to stabilize the voltage within a range. Also, because of the different approximation models, the theoretical guarantees of \cite{zhang2013local} and our paper are slightly different. For instance, they show that the convergence depends on the system operating condition, whereas we show that the convergence is regardless of the operating condition. Though we apply the full nonlinear AC model to simulate our algorithms and the simulation results are consistent with the theoretical analysis, it remains as an interesting question about how to choose the power flow approximation model and how good the approximation is.} 
 Another similar paper \cite{kam2014stability} also proposes a linear decentralized control algorithm for voltage control using a similar power flow approximation model to ours. However, their control algorithm only uses voltage measurement, making the algorithm fall into the category of our algorithm 0, which is insufficient if the objective is to stabilize the voltage within the acceptable range. 

In addition to the work mentioned above, in power engineering community, much work has been done in voltage control for microgrids \cite{savaghebi2012secondary,aalborg2012droop,
aalborg2013secondary,rocabert2012microgrid,
vandoorn2012microgrid,hu2011droop,savaghebi2012secondary}. See \cite{ali2012microgrid} for a comprehensive review. The existing methods generally fall under two layers, primary control (droop control) and secondary control.\footnote{There is usually a third layer of microgrid control that focuses on the economical operation of microgrid. This layer is beyond the interest of the discussion here.} The droop control can be viewed as a linear controller in which the reactive power injection is a linear function of the measured voltage. In section \ref{sec:local}, we will show that this kind of methods do not work under some operation conditions. The secondary control is essentially an integral controller that eliminates the deviation between the voltage and a reference point. This method can be viewed as a special case of algorithm 1 in this paper by shrinking the acceptable range to the reference point. 

The remaining of this paper is organized as follows: Section~\ref{sec:problem} presents an AC power flow model, its linear approximation, and the formulation of the Volt/Var control; Section~\ref{sec:prob} illustrates the impossibility result of merely using voltage information in the control; Section~\ref{sec:dec_feasible} presents one decentralized robust algorithm to maintain acceptable voltages; Section~\ref{sec:dec_opt} presents one decentralized robust algorithm to maintain acceptable voltages and also reach certain optimality as to the reactive power support; Section~\ref{sec:case} simulates the algorithms to complement our analysis; Sectoin~\ref{sec:conclusion} concludes the paper.


\section{Preliminaries: Power Flow Model and Problem Formulation}\label{sec:problem}

Due to space limit, we introduce here an abridged version of the branch
flow model; see, e.g., \cite{Farivar-2012-BFM-TPS,Gan2012branch}
for more details.

\subsection{Branch flow model for radial networks\label{sub:branchflow}}

Consider a radial distribution circuit that consists of a set $N$
of buses and a set $E$ of distribution lines connecting these buses.
We index the buses in $N$ by $i=0,1,\dots,n$, and denote a line
in $E$ by the pair $(i,j)$ of buses it connects. Bus $0$ represents
the substation and other buses in $N$ represent branch buses. For
each line $(i,j)\in E$, let $I_{i,j}$ be the complex current flowing
from buses $i$ to $j$, $z_{ij}=r_{ij}+\textbf{i}x_{i,j}$
be the impedance on line $(i,j)$, and $S_{ij}=P_{ij}+\textbf{i}Q_{i,j}$ be
the complex power flowing from buses $i$ to bus $j$. On each bus
$i\in N$, let $V_{i}$ be the complex voltage and $s_{i}=p_i + \textbf{i} q_i$ be the
complex power injection, i.e., the generation minus consumption. As customary,
we assume that the complex voltage $V_{0}$ on the substation bus
is given and fixed at the nominal value.

The branch flow model was first proposed in \cite{Baran1989a,Baran1989b}
to model power flows in a steady state in a radial distribution circuit: 

\begin{subeqnarray}
\label{eq:branchflow_AC}
-p_j & = &  P_{ij} - r_{ij} \ell_{ij} - \sum_{k: (j,k)\in E} P_{jk},
~j=1,\ldots,n
\label{eq:Kirchhoff.2a}
\\
-q_j & = &  Q_{ij} - x_{ij} \ell_{ij}  - \sum_{k: (j,k)\in E} Q_{jk},
~j=1,\ldots,n
\label{eq:Kirchhoff.2b}
\\
v_j & = & v_i - 2 (r_{ij} P_{ij} + x_{ij} Q_{ij}) + (r_{ij}^2 + x_{ij}^2) \ell_{ij},
\nonumber
\\
& & \ \ \ \ \ \ \ \ \  \ \ \ \ \ \ \ \ \ \ \ \ \ \ \ \ 
(i, j)\in E
\label{eq:Kirchhoff.2c}
\\
\ell_{ij} & = &  \frac{P_{ij}^2 + Q_{ij}^2}{v_i},
\ \ \ \ \ \ \ \ \ \ \  \ \ (i, j)\in E,
\label{eq:Kirchhoff.2d}
\end{subeqnarray}
where $\ell:=|I_{ij}|^{2}$, $v_{i}:=|V_{i}|^{2}$. Equations
(\ref{eq:branchflow_AC}) 
define a system of equations in the variables $(P,Q,\ell,v):=(P_{ij},Q_{ij},\ell_{ij},(i,j)\in E,i=1,\dots,n)$,
which do not include phase angles of voltages and currents. Given
an $(P,Q,\ell,v)$ these phase angles can be uniquely determined for
radial networks. This is not the case for mesh networks; see \cite{Farivar-2012-BFM-TPS}
for exact conditions under which phase angles can be recovered for mesh networks.

\subsection{Linear approximation of the branch flow model}
Real distribution circuits usually have very small $r,x$, i.e. $r,x<<1$,
while $v\sim 1$. Thus real and reactive power losses are typically much smaller than power flows $P_{ij},Q_{ij}$. Following\cite{Baran1989c}, we neglect the higher order real and reactive power loss terms in (\ref{eq:branchflow_AC}) by setting $\ell_{ij}=0$ and approximate $P,Q,v$ using the following linear approximation, known as Simplified Distflow introduced in \cite{Baran1989c}. 
\begin{subeqnarray}
\label{eq:branchflow_linear}
-p_j & = &  P_{ij} - \sum_{k: (j,k)\in E} P_{jk},
~j=1,\ldots,n
\\
-q_j & = &  Q_{ij} - \sum_{k: (j,k)\in E} Q_{jk},
~j=1,\ldots,n
\\
v_j & = & v_i - 2 (r_{ij} P_{ij} + x_{ij} Q_{ij}), 
(i, j)\in E
\end{subeqnarray} 

From (\ref{eq:branchflow_linear}), we can derive that the voltage $v=(v_1,\ldots,v_n)^T$ and power injection $p=(p_1,\ldots,p_n),q=(q_1,\ldots,q_n)$ satisfy the following equation: 
\begin{equation}
v=Rp+Xq+v_{0}
\label{eq:v_pq}
\end{equation}
where $R=\left[R_{ij}\right]_{n\times n}$, $X=\left[X_{ij}\right]_{n\times n}$ are given as follows: 
\begin{eqnarray*}
R_{ij}:=2\sum_{(h,k)\in \mathcal{P}_i \cap \mathcal{P}_j} r_{hk}, \\
X_{ij}:=2\sum_{(h,k)\in \mathcal{P}_i \cap \mathcal{P}_j} x_{hk}.  \\
\end{eqnarray*}
Here $\mathcal{P}_i \subset E$ is the set of lines on the unique path from bus $0$ to bus $i$. The detailed derivation is given in \cite{lijun2013voltage}. Since $r_{ij}>0, x_{ij}>0$ for all $i,j$, $R, X$ have the following properties.

\begin{lemma}
R, X are positive definite and positive matrices.\footnote{A matrix is a positive matrix iff each item is positive.}
\end{lemma}
\begin{proof}
We refer readers to \cite{lijun2013voltage} for the detailed proof.
\end{proof}

\subsection{Problem formulation}\label{sec:prob}
Before rigorously formulating the Volt/Var control problem, we separate $q$ into two parts, $q^c=(q_1^c,\ldots,q_n^c)$ and $q^{e}=(q_1^e,\ldots,q_n^e)$, where $q^c$ denotes the reactive power injection governed by the Volt/Var control components and $q^e$ denotes any other reactive power injection.\footnote{For easy exposition, we assume that there is a $q^c_i$ at each bus $i$. But the algorithm extends to the scenario that only a subset of buses have Volt/Var control component.} Let $v^{par}\triangleq Rp+Xq^e+v_0$, then,
$$v= Xq^c+v^{par}.$$

The goal of Volt/Var control on a distribution network is to provision reactive power injections $q^c$ to maintain the bus voltages $v$ within a tight range $[\underline {v}, \bar{v}]$ under any operating condition given by $v^{par}$. Without causing any confusion, in the rest of the paper, we will simply use $q_i$  instead of $q_i^c$ to denote the reactive power pulled by the Volt/Var control devices. The Volt/Var control can be modeled as a control problem on a quasi-dynamical system with state $v$ and controller $q$; that is, given the current state $v(t)$ and other available information, the controller determines a new reactive power injections $q(t)$ and the new $q(t)$ results in a new voltage profile $v(t+1)$ according to (\ref{eq:v_pq}). 
Mathematically, the Volt/Var control problem is formulated as the following closed loop dynamical system,
\begin{subeqnarray} \label{eq:dynamicsystem}
v(t+1) & = & Xq(t)+v^{par}; \slabel{eq:prob_v}\\
q(t) & = & u(\text{information at time } t). \slabel{eq:prob_q}
\end{subeqnarray}
where $u=(u_1, \ldots,u_n)$ is the Volt/Var controller. The objective of Volt/Var control is to design $u$ to drive the system voltage $v(t)$ to  the acceptable range $ [\underline{v},\bar{v}]$ under any system operating condition which is given by $v^{par}$. Rigorously speaking, it requires that 
$$\lim_{t\rightarrow \infty}\text{dist}(v(t),[\underline{v},\bar{v}])=0.$$
Here $\text{dist}(y,Z):=\min_{z\in Z}||y-z||$ where $y$ is a point and $Z$ is a set. Note that (\ref{eq:prob_v}) is governed by the system intrinsic dynamics (Kirchoff's Law) and not able to controlled or tuned, which makes the Volt/Var control  challenging.

The problem we address in this paper is to design voltage control algorithm $u$ where each $u_i$ only uses local information, as shown in Figure~\ref{fig:substation}. In general, in electricity distribution networks, only a small portion of buses are monitored, individuals are unlikely to announce their generation or load profile, the availability of DERs are fluctuating and uncertain, and even the grid parameters and the topology are only partially known. All of these facts make decentralized algorithms necessary for the voltage control, i.e., each control component adjusts its reactive power input based on the local signals that are easy to measure or to communicate.

\section{Voltage control using only voltage measurements: Impossibility result}\label{sec:local}
 
We first study such Volt/Var control rules that use merely voltage measurements as the control information. This type of control has been proposed and discussed in many existing literature and applications 
\cite{ali2012microgrid}. For example, the IEEE 1547.8 standard \cite{IEEE1547} proposes decentralized voltage control for inverters using the deviation of the local voltages from the nominal value: an inverter monitors its terminal voltage and sets its reactive power generation based on a static and predefined Volt/Var curve. If the predefined Volt/Var curve is linear, the controller is called a droop controller. Regarding a detailed theoretic analysis of droop controllers, please see \cite{kam2014stability}. The scheme of such controllers is shown in Figure \ref{fig:v_dec}. Besides the IEEE 1547.8 standard, there are other researches promoting adapting reactive power injection according to the voltage \cite{carvalho2013distributed,kam2014stability,aalborg2012droop}. However, in the following, we show that this type of controller is insufficient to drive the voltage to the acceptable range even if controller stabilizes the system.
\begin{figure}[ht]
\centering
\includegraphics[width=2in]{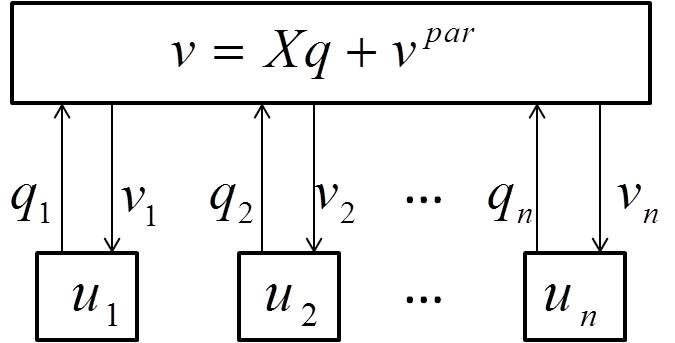}
\caption{IEEE 1547.8 standard: Decentralized Volt/Var control using local voltage measurements.}\label{fig:v_dec}
\end{figure}

\begin{figure}[ht]
\centering
\includegraphics[width=2in]{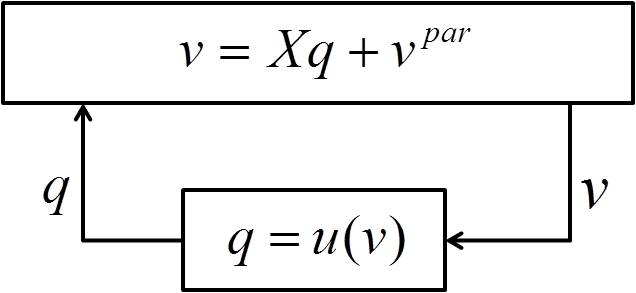}
\caption{Volt/Var control using merely voltage information: This type of controller is insufficient for Volt/Var control.}\label{fig:v_gen}
\end{figure}

In fact, we will demonstrate that as long as $u$ is in the following form,
\begin{equation}
q(t)=u(v(t)), \label{eq:puv}
\end{equation}
and $u$ maps the bounded set $[\underline{v},\bar{v}]$ to a bounded set, then it is impossible for such $u$ to maintain acceptable voltages under all the operating condition, no matter whether $u$ is in a centralized or decentralized form. This is formally stated in the following proposition. 
\begin{proposition}\label{prop:volt}
For any $u$ in the form of (\ref{eq:puv}) that maps $[\underline{v},\bar{v}]$ to a bounded set, there exist $v^{par}$ such that this controller is not able to stabilize the voltage $v$ in the acceptable range $[\underline{v},\bar{v}]$.\footnote{Note that any continuous function maps $[\underline{v},\bar{v}]$ into a bounded set.} 
\end{proposition}
\begin{proof}
The proof is straghtforward. Subsistuting (\ref{eq:puv}) into (\ref{eq:prob_v}), we have:
$$
v(t+1)-Xu(v(t))=v^{par}.
$$
Given a $v^{par}$, if $u$ stabilizes the voltage in the acceptable range $[\underline{v},\bar{v}]$, i.e., $\lim_{t\rightarrow \infty}\text{dist}(v(t),[\underline{v},\bar{v}])=0$, then $v^{par}$ should be at least in the set of $M:=\{v-Xu(\tilde{v}):v,\tilde{v}\in [\underline{v},\bar{v}]\}$ which is bounded because $u$ maps $[\underline{v},\bar{v}]$ to a bounded set.  Thus we know there exist $v^{par}$ that the controller is not able to stabilize the voltage in the acceptable range. 
\end{proof}

This proposition tells us that almost any Volt/Var control  that depends merely on voltage information is not suitable to maintain acceptable voltages, no matter whether it is decentralized or centralized. Thus we should consider adding (or using) other information to design controller $u$. In the rest of the paper, we will show that if we use both the information of the current $q$ and $v$, then a fully decentralized algorithm in the form of $u_i(v_i(t), q_i(t-1))$ is able to maintain acceptable voltages; further, if we introduce some auxiliary variables $\lambda_i, \forall i $, then a fully decentralized algorithm in the form of $u_i\left(v_i(t), \lambda_i(t-1)\right)$ is able to both maintain acceptable voltages and minimize a cost of reactive power compensation in a particular form.


\section{A decentralized algorithm to reach the acceptable volage range}\label{sec:dec_feasible}
In this section, we will show that by using both the local voltage and reactive power information, a fully decentralized algorithm in the form of $q_i(t)=u(v_i(t),q_i(t-1))$ is able to stabilize the voltage within the acceptable range. The scheme of the algorithm is shown in Figure~\ref{fig:qv_dec}. Moreover, this algorithm is actually robust to heterogeneous update rates and delays. Firstly, each control device is free to choose when to update the reactive power injection; secondly, it is allowed to have measurement and actuation delay.  This robustness makes the algorithm suitable for practical applications because in practice different buses may have different control devices which operate at different time rates and delay widely exists in real systems.

We now present the first algorithm which we name as algorithm 1 and then prove its convergence.
\begin{figure}[ht]
\centering
\includegraphics[width=2in]{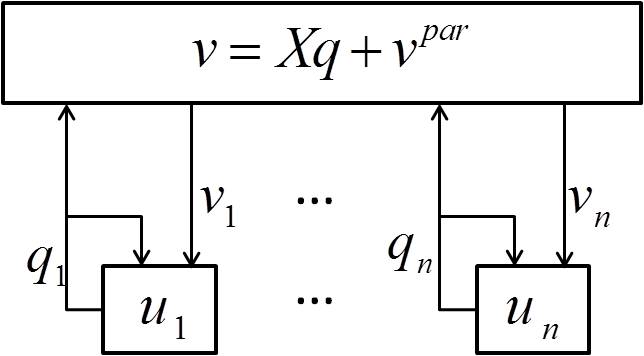}
\caption{Volt/Var control using local information of local voltage and reactive power injection: This type of controller is able to maintain acceptable voltages.}\label{fig:qv_dec}
\end{figure}

\subsection{Algorithm 1}
For each bus $i$, we introduce a infinite time step set $T_i \subseteq \{1,2,\ldots\}$ to represent the time steps when bus $i$ updates the reactive power rejection. 

For $t \notin T_i$, no new decision will be made, which means the reactive injection stays the same as the previous time step, i.e.
\begin{align}
\label{eq:asyn_noupdate_1}
q_i(t)&= q_i(t-1)
\end{align}

For $t \in T_i$, bus $i$ updates according to the following update rule,
\begin{align}
\label{eq:asyn_update_1}
q_i(t)&= q_i(t-1) - \epsilon d_i(v_i(t - \tau_i(t)))
\end{align}
where given a voltage $v_i$, 
\begin{eqnarray}
d_i(v_i):=\left[v_i-\bar{v}_{i}\right]^+-\left[\underline{v}_{i}-v_i\right]^+
\end{eqnarray} and $\epsilon$ is a positive constant stepsize. Here $[]^+$ is defined as $[a]^+:=\max(a,0)$. $\tau_i(t)\geq 0 $ is the measurement delay of bus $i$ at $t$. $v_i(t - \tau_i(t))$ represents the voltage measurement bus $i$ received from a voltage sensor at time $t$, which is the actual voltage $\tau_i(t)$ steps before $t$.\footnote{When bus $i$ is not updating at $t$, we would abuse the notation $\tau_i(t)$ to be an arbitrary nonnegative integer despite it will have no effect on the algorithm.} 


The algorithm in (\ref{eq:asyn_update_1}) says that if the voltage at bus $i$ exceeds its upper limit, bus $i$ decreases the reactive power injection; in contrast, if the local voltage falls below its lower limit, bus $i$ increases its reactive power injection. This algorithm is very simple and intuitive, yet we will prove that regardless of $v^{par}$ and despite the heterogeneous update rates and the delays, the voltage will asymptotically converge to the acceptable range $[\underline{v},\bar{v}]$ under the controller specified in (\ref{eq:asyn_noupdate_1}) and (\ref{eq:asyn_update_1}). Before that, we discuss the properties of this algorithm which make it attractive to real-time and scalable implementation. 
\begin{enumerate}

\item
We note that in the algorithm, each bus $i$ uses only the local voltage measurement $v_i(t - \tau_i(t))$ and its previous reactive power injection $q_i(t-1)$. This algorithm is fully decentralized, requiring no communication. 
\item The controller (\ref{eq:asyn_update_1}) is similar to an integral controller. The implementation is simple. 
\item The algorithm does not require any information about the system operating condition, represented by$v^{par}$. This makes the algorithm practical because due to the large volatility and uncertainty of renewable energy, time-varying nature of demand, and the privacy concern of consumers, $v^{par}$ is not available and fluctuates fast and by a large amount.  
\item Though the convergence of the algorithm depends on the value of the step size $\epsilon$, as shown in the next section, the criterion for choosing $\epsilon$ to guarantee the convergence is independent of $v^{par}$. As a result, once we have incorporated the algorithm into the hardware design of Volt/Var control, the Volt/Var control will work under any system operating operation.
\end{enumerate}
Before we proceed to the convergence result, we make the following assumption on the update time steps $T_i$ and the measurement delays $\tau_i(t)$.
\begin{assumption}\label{assump:1}
For any $i$, the difference between consecutive update time steps in $T_i$ is upper bounded by $T_a$. For any $i$ and $t$, the time delay $\tau_i(t)$ is upper bounded by $T_d$. 
\end{assumption}
The first part of the assumption means that each bus should update at least once within $T_a$ steps; the second part means that the delays are uppder bounded by $T_d$. This assumption is mild and reasonable, and can be easily guaranteed in practice. Under this assumption, we have the following convergence result.

\begin{theorem} \label{thm:asyn_1}
For Algorithm 1 in (\ref{eq:asyn_noupdate_1}) and (\ref{eq:asyn_update_1}), if Assumption \ref{assump:1} holds and $\epsilon < \frac{2}{\sigma_{max}(X) +2\Vert X\Vert_F T_d}$, then $d(v(t))$ will converge to $0$, and $v(t)$ will converge to a point $v^*\in [\underline{v},\bar{v}]$. 
\end{theorem}

From Theorem \ref{thm:asyn_1} we can see that the maximum delay time $T_d$ contributes a linear term in the denominator of the maximum stepsize that guarantees the convergence. The slope of the linear term is the Frobenius norm of $X$, $\Vert X\Vert_F$. The theorem shows that the larger $\Vert X\Vert_F$ is (which is usually a result of a larger network) or the larger $T_d$ is, the smaller the stepsize should be. 
\subsection{Proof of the convergence}
In this section, we prove that the voltage asymptotically converges to the acceptable range $[\underline{v},\bar{v}]$ under the controller specified in (\ref{eq:asyn_noupdate_1}) and (\ref{eq:asyn_update_1}). We first introduce a Lyapunov function $\psi$, which will decrease along the trajectory of the control algorithm. Then, using a similar technique as in \cite{low1999optimization}, we will show that $\Vert q(t) - q(t-1)\Vert \rightarrow 0$. Combining this with the fact that each bus has to update at least once within $T_a$ steps, we will reach the conclusion that $d(v(t))$ asymptotically converges to $0$.

We first introduce some notations. Let $\alpha(t) = q(t) - q(t-1)$. Define a $n$-vector $\hat{d}(v) = v - \frac{\bar{v}+\underline{v}}{2}$, i.e. its $i$th component is $\hat{d}_i(v_i) = v_i - \frac{\bar{v}_i+\underline{v}_i}{2}$. Define a Lyapunov function $\psi(q) = \frac{1}{2}\hat{d}(v(q))^T X^{-1}\hat{d}(v(q))$ where $v(q) =Xq+v^{par}$. Clearly, $\nabla \psi(q) = \hat{d}(v(q))$ and $\nabla^2 \psi(q) = X$. Define $\gamma(t)$ to be a $n$-vector with its $i$th component as $\gamma_i(t) = \hat{d}_i(v_i(t-\tau_i(t) ))$.

\begin{lemma} \label{lem:gamma_psi}
$\Vert \gamma(t) - \nabla \psi(q(t-1))\Vert \leq \Vert X\Vert_F \sum_{t' = t-T_d}^{t-1} \Vert \alpha(t')\Vert$
\end{lemma}
\begin{proof}
For each $i$, 
\begin{align*}
|\gamma_i(t) - \nabla_i \psi(q(t-1))| &= |\hat{d}_i(v_i(t-\tau_i(t) )) - \hat{d}_i(v_i(t))|\\
                                 &= |v_i(t-\tau_i(t) ) - v_i(t)| \\
                                 &= |X_i^T (q(t-1) - q(t-1-\tau_i(t)))|\\
                                 &\leq \Vert X_i \Vert \Vert q(t-1) - q(t-1-\tau_i(t))\Vert\\
                                 &= \Vert X_i \Vert \Vert \sum_{t' = t-\tau_i(t)}^{t-1} \alpha(t') \Vert\\
                                 &\leq \Vert X_i \Vert \sum_{t' = t-\tau_i(t)}^{t-1} \Vert \alpha(t')\Vert\\
                                 &\leq \Vert X_i \Vert \sum_{t' = t-T_d}^{t-1} \Vert \alpha(t')\Vert
\end{align*}
where $X_i$ is the $i$th row of $X$. Sum over $i$ and the lemma follows.
\end{proof}
\begin{lemma}\label{lem:gamma_alpha}
$$\gamma(t)^T \alpha(t) = -\frac{1}{\epsilon} \Vert \alpha(t)\Vert^2 - \sum_{i=1}^n M_i |\alpha_i(t)|$$
where $M_i:=\frac{\bar{v}_i-\underline{v}_i}{2}$.
\end{lemma}
\begin{proof}
For each i, $\gamma_i(t) =\hat{d}_i(v_i(t-\tau_i(t)))$. For notational simplicity, we denote $t-\tau_i(t)$ by $t'$. If bus $i$ updates at $t$, then $\alpha_i(t) = -\epsilon d_i(v_i(t'))$. If $d_i(v_i(t'))>0$, then $\hat{d}_i(v_i(t'))= d_i(v_i(t')) + \frac{\bar{v}_i - \underline{v}_i}{2}$ and $\hat{d}_i(v_i(t'))d_i(v_i(t'))= d_i(v_i(t'))^2 + \frac{\bar{v}_i - \underline{v}_i}{2} d_i(v_i(t'))$, i.e. $\gamma_i(t)\alpha_i(t) = (-\frac{1}{\epsilon})\alpha_i(t)^2 - \frac{\bar{v}_i - \underline{v}_i}{2} |\alpha_i(t)|$. If $d_i(v_i(t'))<0$, then $\hat{d}_i(v_i(t'))= d_i(v_i(t')) - \frac{\bar{v}_i - \underline{v}_i}{2}$ and hence $\hat{d}_i(v_i(t'))d_i(v_i(t')) = d_i(v_i(t'))^2 - \frac{\bar{v}_i - \underline{v}_i}{2}d_i(v_i(t')) $, i.e. $\gamma_i(t)\alpha_i(t)= (-\frac{1}{\epsilon})\alpha_i(t)^2 - \frac{\bar{v}_i - \underline{v}_i}{2}|\alpha_i(t)|$. If $d_i(v_i(t')) =0$ or bus $i$ does not update, we also have, $\gamma_i(t)\alpha_i(t)= (-\frac{1}{\epsilon})\alpha_i(t)^2 - \frac{\bar{v}_i - \underline{v}_i}{2}|\alpha_i(t)|$ because both sides of the equality are zero. Sum the equality over $i$,
\begin{align*}
\gamma(t)^T \alpha(t) &= \sum_{i=1}^n \Big[-\frac{1}{\epsilon}\alpha_i(t)^2 - \frac{\bar{v}_i - \underline{v}_i}{2}|\alpha_i(t)|\Big]\\
                      &= -\frac{1}{\epsilon} \Vert \alpha(t)\Vert^2 - \sum_{i=1}^n M_i |\alpha_i(t)|
\end{align*}
\end{proof}
Now we are ready to prove Theorem \ref{thm:asyn_1}.
\begin{proof}
With Lemma \ref{lem:gamma_psi} and Lemma \ref{lem:gamma_alpha} we have
\begin{align*}
& \psi(q(t)) - \psi(q(t-1)) \\
 \leq&\nabla \psi(q(t-1))^T \alpha(t) + \frac{1}{2} \alpha(t)^T \nabla^2 \psi(q(t-1)) \alpha(t) \\
                           \leq& \Vert\nabla \psi(q(t-1))-\gamma(t)\Vert \Vert \alpha(t)\Vert+\gamma(t)^T\alpha(t) + \frac{1}{2} \alpha(t)^T X \alpha(t)\\
                          \leq& \Vert X\Vert_F \sum_{t' = t-T_d}^{t-1} \Vert \alpha(t')\Vert \Vert \alpha(t)\Vert \\
                          &- \sum_{i=1}^n M_i |\alpha_i(t)| + (-\frac{1}{\epsilon} +\frac{1}{2} \sigma_{max}(X)) \Vert \alpha(t)\Vert^2 \\
                          \leq& \frac{1}{2}\Vert X\Vert_F \sum_{t' = t-T_d}^{t-1} (\Vert \alpha(t')\Vert^2+ \Vert \alpha(t)\Vert^2)  \\
                          &- \sum_{i=1}^n M_i |\alpha_i(t)|+(-\frac{1}{\epsilon} + \frac{1}{2} \sigma_{max}(X))\Vert \alpha(t)\Vert^2                          
\end{align*}

Sum from $t=1$ to $t=T$, we have
\begin{align}
&\psi(q(T)) - \psi(q(0))\nonumber\\
 \leq& \frac{1}{2}\Vert X\Vert_F \sum_{t=1}^{T}\sum_{t' = t-T_d}^{t-1} (\Vert \alpha(t')\Vert^2+ \Vert \alpha(t)\Vert^2) \nonumber\\
 & +(-\frac{1}{\epsilon} + \frac{1}{2} \sigma_{max}(X))\sum_{t=1}^{T}\Vert \alpha(t)\Vert^2 - \sum_{t=1}^{T}\sum_{i=1}^n M_i |\alpha_i(t)|\nonumber\\
                          =& \frac{1}{2}\Vert X\Vert_F \sum_{t=1}^{T}\sum_{t' = t-T_d}^{t-1} \Vert \alpha(t')\Vert^2- \sum_{t=1}^{T}\sum_{i=1}^n M_i |\alpha_i(t)|\nonumber \\
                          &+  (-\frac{1}{\epsilon} + \frac{1}{2} \sigma_{max}(X) +\frac{T_d}{2}\Vert X\Vert_F)\sum_{t=1}^{T}\Vert \alpha(t)\Vert^2\nonumber\\                                                    
                          \leq& (-\frac{1}{\epsilon} + \frac{1}{2} \sigma_{max}(X) + T_d\Vert X\Vert_F)\sum_{t=1}^{T}\Vert \alpha(t)\Vert^2 \nonumber\\
                          &- \sum_{t=1}^{T}\sum_{i=1}^n M_i |\alpha_i(t)| \label{eq:upperbound_alpha}
\end{align}

Since $\psi(q(T))$ is lower bounded and $M_i\geq 0$, if $\epsilon<\frac{2}{\sigma_{max}(X) + 2\Vert X\Vert_F T_d}$, $\sum_{t=1}^{T}\Vert \alpha(t)\Vert^2 $ is upper bounded, hence $\alpha(t) \rightarrow 0$. For each $i$, denote the elements in $T_i$ by an increasing sequence $\{t_i(n)\}_{n=1}^{\infty}$. Denote $t'_i(n) =t_i(n)- \tau_i(t_i(n))$, then by (\ref{eq:asyn_update_1}), $d_i(v_i(t'_i(n)))\rightarrow 0$ as $n\rightarrow\infty$. For each $t$, by assumption \ref{assump:1}, $\exists n_t\in \mathbb{N}$ s.t. $t - T_a - T_d \leq t'_i(n_t)\leq t$. So
\begin{align}\label{eq:proof_thm3_d}
|d_i(v_i(t)) - d_i(v_i(t'_i(n_t)))| & \leq | v_i(t) - v_i(t'_i(n_t))| \nonumber\\
                                    & \leq \Vert X_i\Vert \Vert q(t-1) - q(t'_i(n_t)-1)\Vert \nonumber \\
                                    & \leq \Vert X_i\Vert \sum_{t' = t-T_a -T_d }^{t-1} \Vert \alpha(t')\Vert
\end{align}
Where the first inequality follows from the fact that $d_i(v)$ is a Lipshitz function with Lipshitz coefficient $1$. Let $t\rightarrow\infty$, then $n_t\rightarrow \infty$, so $ d_i(v_i(t'_i(n_t)))\rightarrow 0$. Also notice the right hand side of (\ref{eq:proof_thm3_d}) converges to 0 (because it's a finite sum of sequences each of which converges to 0), we have $d_i(v_i(t))\rightarrow 0$. This holds for all $i$, so $\Vert d(v(t))\Vert \rightarrow 0$.

For each $i$, if $M_i >0$ (i.e. $\bar{v}_i > \underline{v}_i$), by (\ref{eq:upperbound_alpha}) $\sum_{t=1}^T |\alpha_i(t)|$ is upper bounded, hence this series converges. This says $|q_i(T+k) - q_i(T)| \leq \sum_{t = T+1}^{T+k} |\alpha_i(t)| \rightarrow 0$ as $T\rightarrow \infty$. So $q_i(t)$ is a Cauchy sequence and it converges to a point $q_i^*$ \cite{rudin1964principles}. For $i$ that $M_i=0$ (i.e. $\bar{v}_i = \underline{v}_i = v_i^*$), we have $\hat{d}_i(v_i) = d_i(v_i)$, so $\hat{d}_i(t)$ converges to $0$, i.e. $v_i(t)$ converges to $v_i^*$. Now we have for each $i$ either $q_i(t)$ converges or $v_i(t)$ converges. Without loss of generality, we assume for $i = 1,\ldots,m$, $q_i(t)$ converges and for $i = m+1,\ldots,n$, $v_i(t)$ converges. Define a vector $z(t):=(v_1(t),\ldots, v_m(t), q_{m+1}(t-1),\ldots, q_n(t-1))^T-(v_1^{par},\ldots,v_m^{par},0,\ldots,0)^T$. Due to the linear relationship between $v(t)$ and $q(t-1)$, we have: 
$$ z(t)= \left( 
\begin{array}{ccc}
I_m & 0 \\
X_m^1 & X_m^2  \end{array} \right)q(t-1):=\Lambda q(t-1) $$
where $I_m$ is the $m$-dimensional identity matrix, $X_m^1$ and $X_m^2$ together make the $m+1$'s through the $n$'s row of $X$. Since $X$ is positive definite, we have $X_m^2$ is invertible, which leads to $\Lambda$ is invertible. Since $z(t)$ converges, we have $q(t)$ converges, and further $v(t)$ converges to a point $v^*$. Since $\Vert d(v(t))\Vert \rightarrow 0$, $v^*$ lies in the acceptable range.
\end{proof} 

\section{A decentralized algorithm to reach an optimal feasible point}\label{sec:dec_opt}
In this section, we will introduce local auxiliary variables and design a new control algorithm, which we name as algorithm 2. Algorithm 2 will not only drive the voltage into an equilibrium point in $[\underline{v},\bar{v}]$, but also guarantee that the equilibrium point will minimize a cost of reactive power provision in a certain form. We first provide the algorithm and then discuss its convergence and the optimality of the equilibrium point. 

As in algorithm 1, we use the set $T_i$ to represent the time steps that bus $i$ updates its reactive power injection. For each bus $i$, we introduce two local auxiliary variables, $\bar{\lambda}_i$ and $\underline{\lambda}_i$. 

At each time $t$, if $t\not\in T_i$, all the variables stay the same as the previous time step:
\begin{subeqnarray}
\label{eq:asyn_noupdate}
\bar{\lambda}_i(t)&=&\bar{\lambda}_i(t-1) \\
\underline{\lambda}_i(t)&=& \underline{\lambda}_i(t-1) \\
q_i(t)&=& q_i(t-1)
\end{subeqnarray}

If $t\in T_i$,
\begin{subeqnarray} \label{eq:asyn_update}
\bar{\lambda}_i(t)&=&\left[\bar{\lambda}_i(t-1)+\epsilon\left(v_i\left(t-\tau_i(t)\right)-\bar{v}_i\right)\right]^+\\
\underline{\lambda}_i(t)&=&\left[\underline{\lambda}_i(t-1)+\epsilon\left(\underline{v}_i-v_i\left(t-\tau_i(t)\right)\right)\right]^+ \\
q_i(t)&=&\underline{\lambda}_i(t-\tau'_i(t))-\bar{\lambda}_i(t-\tau'_i(t))
\end{subeqnarray}
where $\tau_i(t)$ and $\tau'_i(t)$ capture the measurement, computation, and actuation delay.\footnote{Similar to Algorithm 1, for $t\not\in T_i$ we abuse the notation $\tau_i(t)$ and $\tau'_i(t)$ to be any nonnegative integers.} The diagram of the algorithm is shown in Figure~\ref{fig:qve_dec}.
\begin{figure}[ht]
\centering
\includegraphics[width=2in]{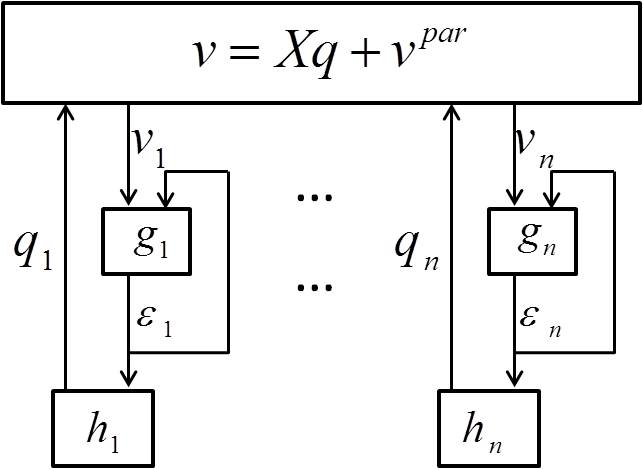}
\caption{Volt/Var control using local information of voltage, reactive power injection, and additional auxiliary variables: This type of controller is able to guarantee both of feasibility and optimality.}\label{fig:qve_dec}
\end{figure}

In this algorithm, the control of the local reactive power depends only on the two local auxiliary variables and the updating rule of the two variables depends on their own values and the current local voltage measurement, requiring no communication and no information about $v^{par}$. Moreover, the updating rule is also similar to an integral controller with saturation. As a result, this algorithm shares the same properties with algorithm 1, making it also attractive to practical applications.

Before we proceed to the convergence analysis of algorithm 2, we point out its relation to the dual gradient algorithm of an optimization problem. To simplify illustration, we consider the following synchronous and delay-free version of algorithm 2:
\begin{subeqnarray}
\label{eq:dual_dec}
\bar{\lambda}_i(t)&=&\left[\bar{\lambda}_i(t-1)+\epsilon(v_i(t)-\bar{v}_i)\right]^+\\
\underline{\lambda}_i(t)&=&\left[\underline{\lambda}_i(t-1)+\epsilon\left(\underline{v}_i-v_i(t)\right)\right]^+ \\
q_i(t)&=&\underline{\lambda}_i(t)-\bar{\lambda}_i(t)
\end{subeqnarray}
We have the following theorem regarding the convergence of algorithm (\ref{eq:dual_dec}):
\begin{theorem}\label{thm:alg2-syn}
If $\epsilon <\frac{1}{\sigma_{\max}(X)}$, $q(t)$ in algorithm (\ref{eq:dual_dec}) converges to the optimal point $q^*$ of the following optimization problem:
\begin{subeqnarray}
\label{eq:min_q}
\min_{q} && \frac{1}{2} q^T X q \\
s.t. &&  Xq+v^{par}\leq \bar{v}, \\
&& Xq+v^{par}\geq \underline{v},
\end{subeqnarray}
and $v(t)$ converges to the corresponding voltage $v^*=Xq^*+v^{par}$, which lies in the acceptable range $[\underline{v},\bar{v}]$.
\end{theorem}
\begin{proof}
Introducing dual variable $\bar{\lambda}$ for (\ref{eq:min_q}b) and $\underline{\lambda}$ for (\ref{eq:min_q}c), we have the following dual gradient algorithm \cite{bertsekas1999nonlinear}:
\begin{eqnarray*}
\bar{\lambda}(t)&=&\left[\bar{\lambda}(t-1)+\epsilon(Xq(t-1)+v^{par}-\bar{v})\right]^+\\
\underline{\lambda}(t)&=&\left[\underline{\lambda}(t-1)+\epsilon\left(\underline{v}-Xq(t-1)-v^{par}\right)\right]^+ \\
q(t)&=&\text{arg}\min_q \{\frac{1}{2}q^T X q  +\bar{\lambda}(t) (Xq+v^{par}-\bar{v}) \\
&&+\underline{\lambda}(t)(\underline{v}-Xq-v^{par})\}
\end{eqnarray*}
The preceding algorithm is equivalent to the following:
\begin{eqnarray*}
\bar{\lambda}(t)&=&\left[\bar{\lambda}(t-1)+\epsilon(v(t)-\bar{v})\right]^+\\
\underline{\lambda}(t)&=&\left[\underline{\lambda}(t-1)+\epsilon\left(\underline{v}-v(t)\right)\right]^+ \\
q(t)&=&\underline{\lambda}(t)-\bar{\lambda}(t)
\end{eqnarray*}
which is exactly the algorithm (\ref{eq:dual_dec}). 

Through simple derivation, the dual problem of (\ref{eq:min_q}) is given by: 
\begin{equation}
\label{eq:dual}
\max_{\bar{\lambda},\underline{\lambda}\geq0} D= -\frac{1}{2} (\underline{\lambda}-\bar{\lambda})^TX (\underline{\lambda}-\bar{\lambda}) +\bar{\lambda}^T(v^{par}-\bar{v})+\underline{\lambda}^T(\underline{v}-v^{par})
\end{equation}

By Lemma~\ref{prop:hessianD} in Appendix, we have $\sigma_{\min}(\nabla^2 D)=-2\sigma_{\max}(X)$. Therefore, we know that if $\epsilon<\frac{1}{\sigma_{\max}(X)}$, $\bar{\lambda}(t), \underline{\lambda}(t)$ converge to the dual optimum. The statement of the theorem then follows. 
\end{proof}

The preceding theorem shows that algorithm 2 is the dual gradient algorithm for (\ref{eq:min_q}) with heterogeneous update rates and delays. Due to this connection, it is intuitive that despite the heterogeneous updates and the delays, algorithm 2 shall still converge, although the convergence might require tighter bounds on $\epsilon$. Before presenting the convergence of algorithm 2, we require the following mild assumption: 
\begin{assumption}\label{assump:2}
$\forall i$, the difference between consecutive update time steps in $T_i$ is upper bounded by $T_a$. $\forall i$ and $\forall t$, $\tau_i(t)$ and $\tau'_i(t)$ are upper bounded by $T_d$. 
\end{assumption}
\begin{theorem} \label{thm:asyn}
In algorithm 2 (\ref{eq:asyn_noupdate}) and (\ref{eq:asyn_update}), if Assumption \ref{assump:2} holds and $\epsilon < \frac{1}{\sigma_{max}(X) +2\Vert X\Vert_F (2T_d+T_a)}$, $q(t)$ and $\lambda(t)$ will converge to the optimizer of the primal problem (\ref{eq:min_q}) and the dual problem (\ref{eq:dual}) respectively, and $v(t)$ will converge to a point $v^*\in[\underline{v},\bar{v}]$.
\end{theorem}

Different from Theorem \ref{thm:asyn_1}, in Theorem \ref{thm:asyn}, $T_a$ contributes a linear term in the denominator of the maximum stepsize. This means that the heterogeneous update rates will have a negative effect on choosing the stepsize of algorithm 2. 


\section{case study}\label{sec:case}
\begin{figure*}[!t]
     	\centering
     	\includegraphics[scale=0.40]{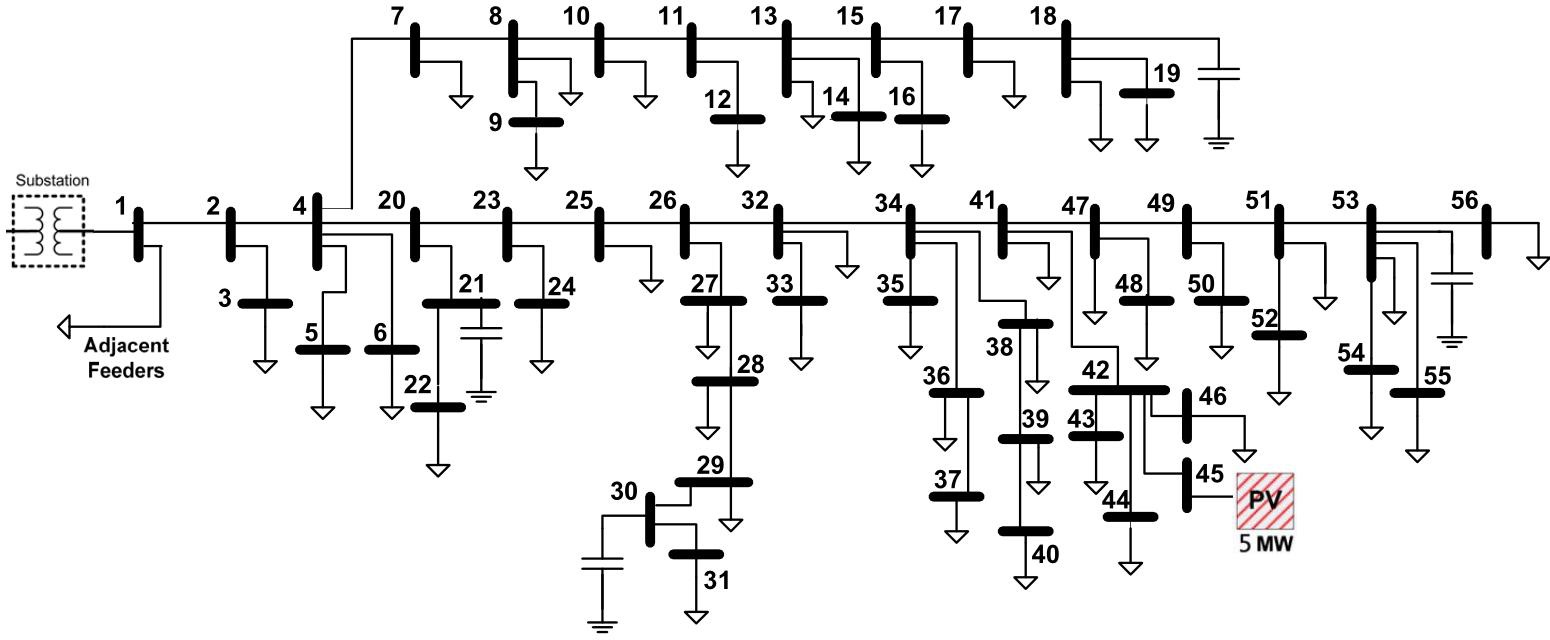}
      	\caption{Schematic diagram of two SCE distribution systems. }
      	\label{fig:circuit}
\end{figure*}

In this section we evaluate algorithm 1 and algorithm 2 on a distribution circuit of South California Edison with a high penetration of photovoltaic (PV) generation
\cite{ Farivar-2012-VVC-PES}.
Figure \ref{fig:circuit} shows a $56$-bus distribution circuit.  
Note that Bus $1$ indicates the substation, and there are 1 photovoltaic (PV) generators located 
at bus $45$ and there are shunt capacitors located at bus $19$, $21$, $30$, $53$. 
See  \cite{Farivar-2012-VVC-PES} for the network data including the line impedance, the peak MVA demand of the loads and the nameplate capacity of the shunt capacitors and the photovoltaic generation. 

In the simulation, we assume that there are Volt/Var control components at bus $19$, $21$, $30$, $45$, and $53$ and those control components can pull in (supply) and out (consume) reactive power. The nominate voltage magnitude is $12 \text{kV}$ and the acceptable range is set as $[11.4\text{kV}, 12.6\text{kV}]$ which is the plus/minus 5\% of the nominate value. Though the analysis of this paper is built on the linearized power flow model (\ref{eq:branchflow_linear}),  we simulate the voltage control algorithms using the full nonlinear AC power flow model (\ref{eq:branchflow_AC}) \cite{5491276}.

We simulate three different scenarios: 1) the PV generator is generating a large amount of power but the loads are moderate, resulting in high voltages at some buses (Figure \ref{fig:asyn_1} and Figure \ref{fig:asyn_2}); 2) the PV generator is generating a large amount of power and some buses are having heavy loads, resulting in high voltages at some buses and low voltages at other buses (Figure \ref{fig:asyn_3} and Figure \ref{fig:asyn_4}). In addition to the above two scenarios, we consider a third scenario where the load and the PV generation are constantly fluctuating, drawn from a uniform distribution (Figure \ref{fig:fluctuating}). The fluctuating speed is considered to be slower than the control algorithm updating speed, which means the controller has a certain amount of time steps to stabilize the voltages after each change in the load and the PV generation. This scenario is introduced to validate the performance of the algorithms under a more realistic setting.  The results of case I and case II are summarized in Figure \ref{fig:asyn_1} to Figure \ref{fig:asyn_4}. In the four figures, the buses update at heterogeneous clock rates and with measurement and actuation delay. Due to the delay, buses might make incorrect updates (e.g., in Figure \ref{fig:asyn_2}, at $t=20$, all the voltages are already within the acceptable range, but the buses drive the voltages outside of the range). However, despite the heterogeneous update clock rates and the delay, the algorithms still converge asymptotically. Figure \ref{fig:fluctuating} verifies the algorithm performance under a fluctuating load and PV generation setting. All the results verify the proposed algorithms' effectiveness. 

\begin{figure}[ht]
\centering
\includegraphics[width=3.7in]{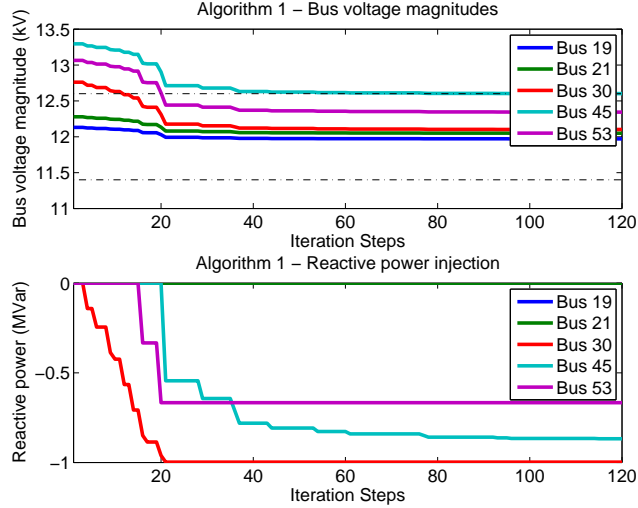}
\caption{Simulation result of case I, algorithm I. The upper figure shows the voltage profile, and the lower figure shows the reactive power injection. Simulation parameters: $T_a = 25$, $T_d = 15$, $\epsilon = 8$.}
\label{fig:asyn_1}
\end{figure}
\begin{figure}[ht]
\centering
\includegraphics[width=3.7in]{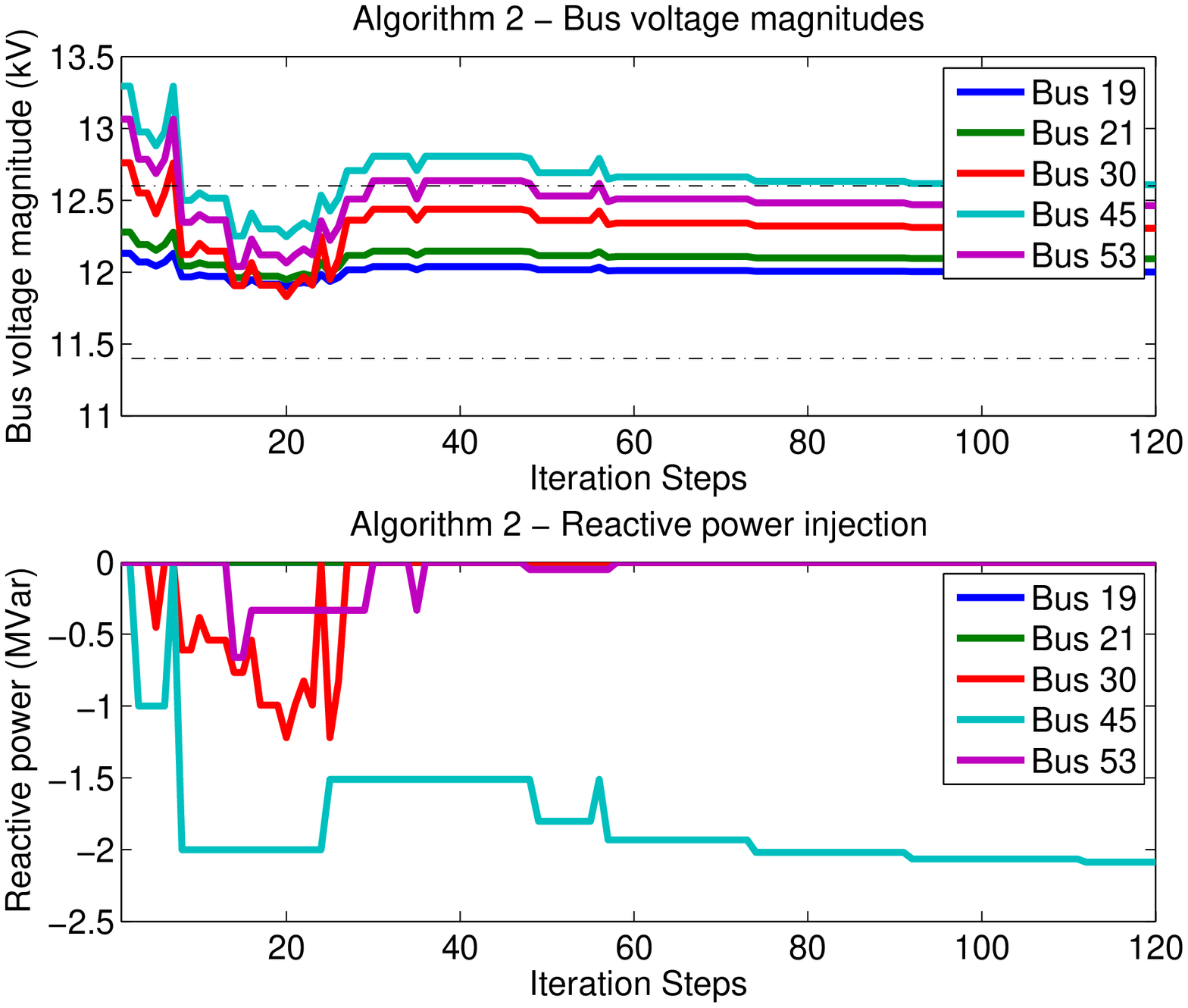}
\caption{Simulation result of case I, algorithm II. The upper figure shows the voltage profile, and the lower figure shows the reactive power injection. Simulation parameters: $T_a = 25$, $T_d = 15$, $\epsilon = 5$.}
\label{fig:asyn_2}
\end{figure}
\begin{figure}[ht]
\centering
\includegraphics[width=3.7in]{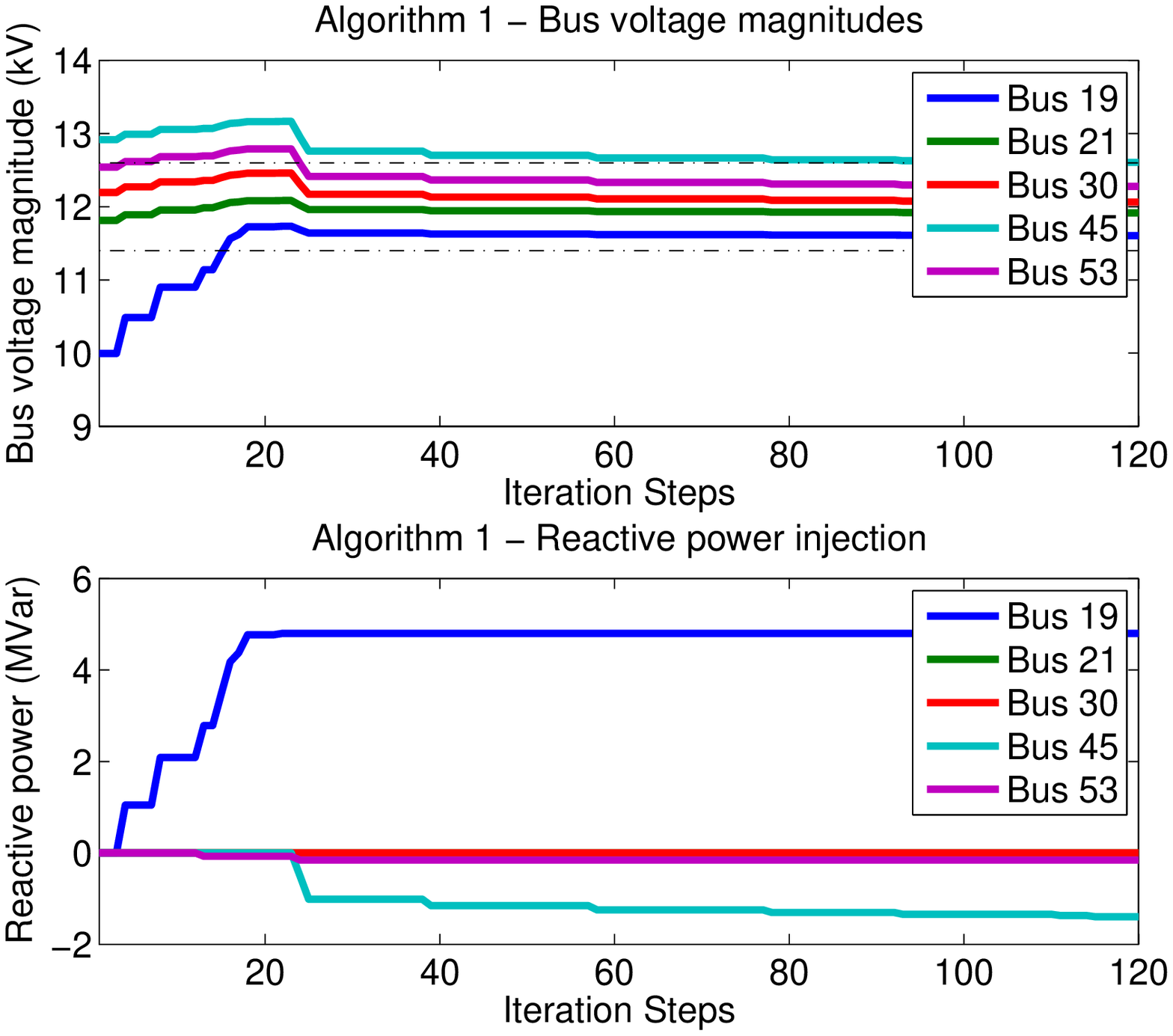}
\caption{Simulation result of case II, algorithm I. The upper figure shows the voltage profile, and the lower figure shows the reactive power injection. Simulation parameters: $T_a = 25$, $T_d = 15$, $\epsilon = 8$.}
\label{fig:asyn_3}
\end{figure}
\begin{figure}[ht]
\centering
\includegraphics[width=3.7in]{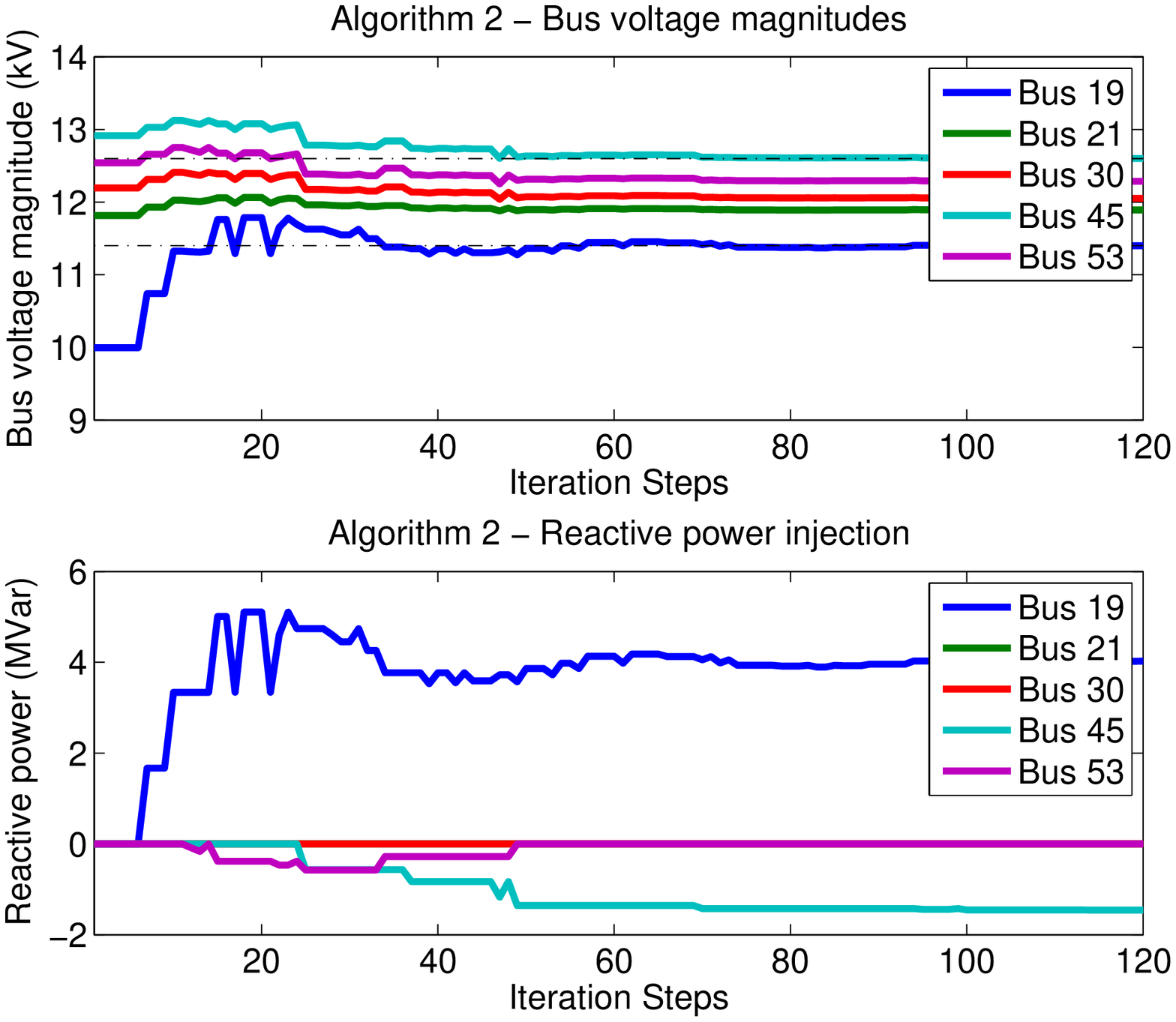}
\caption{Simulation result of case II, algorithm II. The upper figure shows the voltage profile, and the lower figure shows the reactive power injection. Simulation parameters: $T_a = 25$, $T_d = 15$, $\epsilon = 5$.}
\label{fig:asyn_4}
\end{figure}
\begin{figure}[ht]
\centering
\includegraphics[width=3.7in]{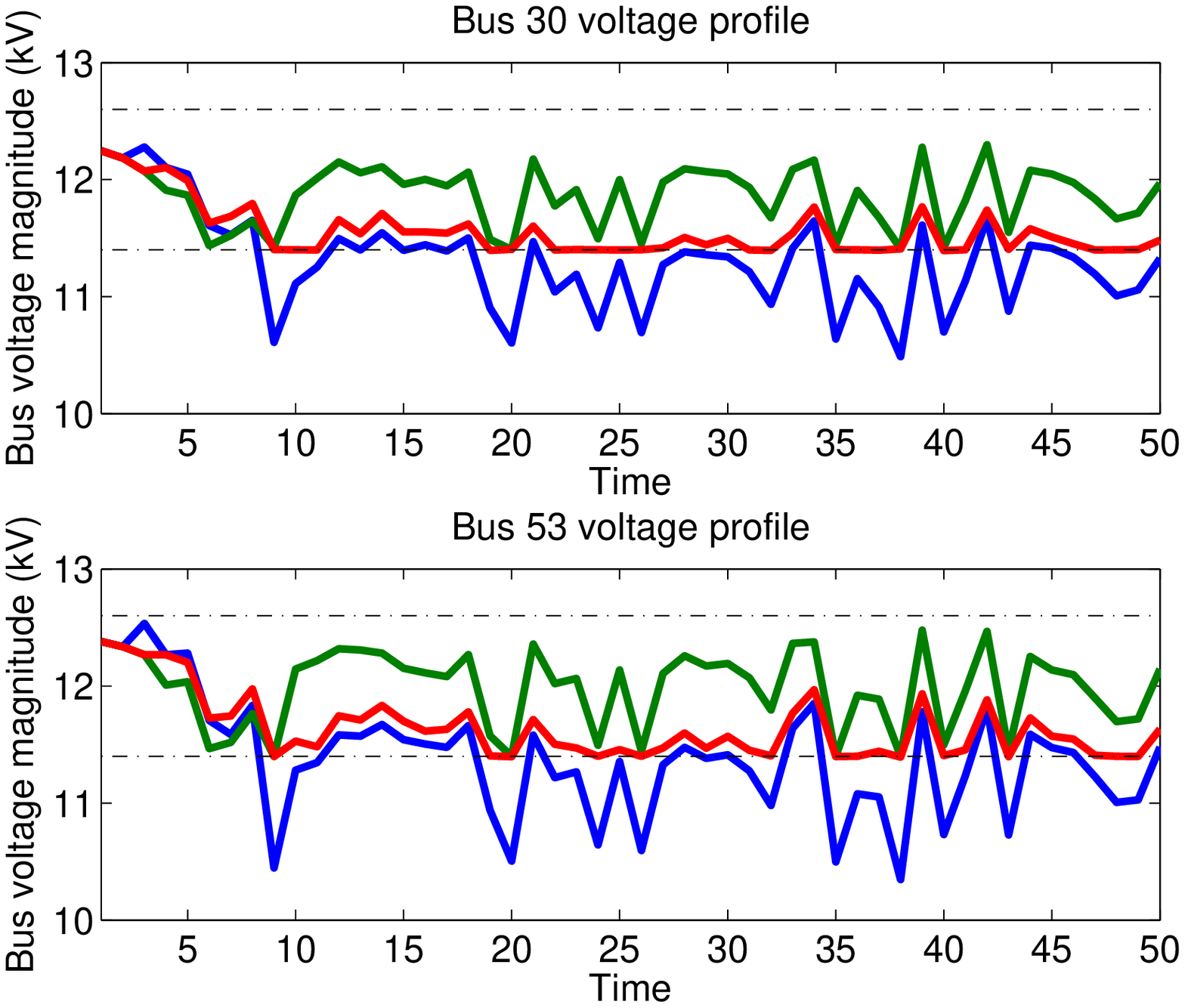}
\caption{Fluctuating load and PV generation. The voltage profile in this figure is the stabilized voltage after each load and generation change. The upper figure shows the voltage profile of bus 30, and the lower figure shows the voltage profile of bus 53. The blue lines are the voltage profile under no reactive power regulation; the green lines algorithm 1; the blue lines algorithm 2.}
\label{fig:fluctuating}
\end{figure}
\section{Conclusion}\label{sec:conclusion}
In this paper, we study how different information structures affect the performance of Volt/Var control. In particular, we first show that using only voltage measurements to decide reactive power injection is insufficient to maintain acceptable voltages. Then we propose two robust and fully decentralized algorithms by adding additional information into the control design. Both of them can maintain acceptable voltages; but one is also able to optimize the reactive power injection in terms of minimizing a cost of the reactive power compensation. Both of the two algorithms use only local measurements and local variables, requiring no communication. They are also robust against heterogeneous update rates and delays. Our results imply that with the aid of right local information, local voltages carry out the whole network information for Volt/Var control.

Regarding future work, there are two additional objectives yet to achieve. First, we have shown that algorithm 2 minimizes a particular form of reactive power cost. However, it would be more desirable if a more general form of cost function can be minimized by the voltage control algorithm. Second, in this paper we do not consider the capacity limit of reactive power injection at each controller, while in practice, we have limit on the reactive power injections, i.e., actuator saturation. It remains as an open question how to design real-time distributed voltage control with provable performance given the capacity limit of reactive power injection. 
\bibliographystyle{IEEETran}
\bibliography{volt_dist}

\begin{appendix}
\subsection{Lemma for Theorem~\ref{thm:alg2-syn}}
\begin{lemma} \label{prop:hessianD}
Regarding the Hessian matrix of $D$ we have (i)$\nabla D^2 = -A$ where $ A = \left( \begin{array}{ccc}
X & -X \\
-X & X  \end{array} \right) $ and (ii)
(ii)$\sigma_{\max}(A) = 2\sigma_{\max}(X)$.
\end{lemma}
\begin{proof}
It's easy to verify (i). For (ii), let $u = (a, b)^T$ where $a$,$b \in \mathbb{R}^n$ and $\Vert u \Vert_2 = 1$. Observe that $\Vert A u \Vert_2^2 = 2 \Vert X a - X b \Vert_2^2 \leq 4 (\Vert Xa \Vert_2^2 + \Vert Xb \Vert_2^2) \leq 4 \sigma_{\max}(X)^2(\Vert a \Vert_2^2 + \Vert b \Vert_2^2) = 4 \sigma_{\max}(X)^2$. Therefore, $\sigma_{\max}(A) \leq 2 \sigma_{\max}(X)$. Let $p$ be the vector s.t. $\Vert p \Vert_2 = 1$ and $\Vert X p \Vert_2 = \sigma_{\max}(X)$. Let $a = -b = p/\sqrt{2}$. Then $\Vert u\Vert_2 =1$ and $\Vert A u \Vert_2 = \Vert \sqrt{2}(X p,X p)^T \Vert_2 = 2\sigma_{\max}(X)$. Hence $\sigma_{\max}(A) = 2\sigma_{\max}(X)$.
\end{proof}

\subsection{Proof for Theorem~\ref{thm:asyn}}
First we introduce a few notations. Define $\lambda(t)=[\bar{\lambda}(t)^T,\underline{\lambda}(t)^T]^T$, $\bar{\pi}(t) = \bar{\lambda}(t) -\bar{\lambda}(t-1)$, $\underline{\pi}(t) = \underline{\lambda}(t) - \underline{\lambda}(t-1)$, and $\pi(t)=[\bar{\pi}(t)^T,\underline{\pi}(t)^T]^T$. Define $\mu(t)=[\bar{\mu}(t)^T, \underline{\mu}(t)^T]^T$ where $[\bar{\mu}(t)]_i = v_i(t-\tau_i(t))-\bar{v}_i$ and $[\underline{\mu}(t)]_i = \underline{v}_i-v_i(t-\tau_i(t))$ for $i = 1,...,n$. 
\begin{lemma} \label{lem:nablapi}
For each $t$, we have
$$\mu(t)^T \pi(t) \geq \frac{1}{\epsilon} \left\Vert \pi(t) \right\Vert^2.$$
\end{lemma}
\begin{proof}
Denote each component of $\pi(t)$ (and $\lambda(t)$, $\mu(t)$) by subscript $l$. For component $l$, if the update rule is (\ref{eq:asyn_noupdate}), apparently we have $[\mu(t)]_l [\pi(t)]_l \geq \frac{1}{\epsilon} [\pi(t)]_l^2$ because both sides of the inequality is $0$.
If the update rule is (\ref{eq:asyn_update}), we have $[\lambda(t)]_l = [[\lambda(t-1)]_l + \epsilon [\mu(t)]_l]^+$. Apply Projection Theorem and therefore $$([\lambda(t-1)]_l+\epsilon [\mu(t)]_l - [\lambda(t)]_l)([\lambda(t-1)]_l - [\lambda(t)]_l) \leq 0 $$
which yields $[\mu(t)]_l [\pi(t)]_l \geq \frac{1}{\epsilon} [\pi(t)]_l^2$.
Summing this inequality over all components $l$ will lead to the lemma.
\end{proof}
\begin{lemma} \label{lem:Dmuerror} The following inequalities hold:

(i) $$\left\Vert\nabla D(\lambda(t-1)) - \mu(t)\right\Vert \leq 2 \Vert X \Vert_F \sum_{t'=t-2T_d-T_a}^{t-1}\left\Vert\pi(t')\right\Vert$$

(ii) $$\Vert v(t)-v(t-1)\Vert \leq  \sigma_{\max}(X) \sum_{t' = t-T_a-T_d}^{t}( \Vert\bar{\pi}(t')\Vert + \Vert\underline{\pi}(t')\Vert)$$
\end{lemma}
\begin{proof}
(i) Through simple calculation, $$\nabla D(\lambda) =  \left( \begin{array}{ccc}
X(\underline{\lambda} - \bar{\lambda}) + v^{par} - \bar{v}   \\
  -[X(\underline{\lambda} - \bar{\lambda}) + v^{par}] + \underline{v}  \end{array} \right) $$
Therefore, $\nabla D(\lambda(t-1)) - \mu(t)$ is a $2 n$-dimension vector with $[\nabla D(\lambda(t-1)) - \mu(t)]_i = X_i^T(\underline{\lambda}(t-1) - \bar{\lambda}(t-1)) + v_i^{par}  - v_i(t-\tau_i(t))$ and $[\nabla D(\lambda(t-1)) - \mu(t)]_{n+i} =  - [\nabla D(\lambda(t-1)) - \mu(t)]_i$ for $i = 1,...,n$, where $X_i$ denotes the $i$th row of $X$. Then,
\begin{align} \label{eq:nabla_mu_i}
&|[\nabla D(\lambda(t-1)) - \mu(t)]_i| \nonumber \\
=&  |X_i^T(\underline{\lambda}(t-1) - \bar{\lambda}(t-1)) - X_i^T q(t-1-\tau_i(t))| \nonumber\\
                       \leq&  \Vert X_i \Vert \Vert \underline{\lambda}(t-1) - \bar{\lambda}(t-1) - q(t-1-\tau_i(t)) \Vert \nonumber\\
                       \leq &  \Vert X_i\Vert \sum_{t' = t - 2 T_d - T_a}^{t-1} (\Vert\underline{\pi}(t')\Vert + \Vert\bar{\pi}(t')\Vert) \nonumber\\
                       \leq & \sqrt{2}\Vert X_i\Vert \sum_{t' = t - 2 T_d - T_a}^{t-1} \Vert\pi(t')\Vert
\end{align}
The last inequality follows from the fact that for any non-negative real number $a$ and $b$ we have $a + b \leq \sqrt{2(a^2+b^2)}$. The second last inequality follows from the fact that for the $j$th component of $\underline{\lambda}(t-1) - \bar{\lambda}(t-1) - q(t-1-\tau_i(t))$ we have
\begin{align*}
&\underline{\lambda}_j(t-1) - \bar{\lambda}_j(t-1) - q_j(t-1-\tau_i(t)) \\
=& (\underline{\lambda}_j(t-1) -\underline{\lambda}_j(t-1-\tau_i(t)-\zeta_j) )\\
                    & - (\bar{\lambda}_j(t-1) - \bar{\lambda}_j(t-1-\tau_i(t)-\zeta_j)) \\
                   = & \sum_{t' = t-\tau_i(t)-\zeta_j}^{t-1} (\underline{\pi}_j(t') - \bar{\pi}_j(t'))
\end{align*}
where $\zeta_j$ is a result of the actuation delay and the non-update time steps caused by (\ref{eq:asyn_noupdate}). Define a n-dimensional vector $\hat{\underline{\pi}}(t')$ whose $j$th component is $\underline{\pi}_j(t')$ when $t-\tau_i(t)-\zeta_j \leq t' \leq t-1$ and $0$ for other $t'$.  Similarly define $\hat{\bar{\pi}}(t')$ whose $j$th component is $\bar{\pi}_j(t')$ when $t-\tau_i(t)-\zeta_j \leq t' \leq t-1$ and $0$ for other $t'$. By the definition we have $\Vert \hat{\underline{\pi}}(t') \Vert \leq \Vert \underline{\pi}(t') \Vert$ and $\Vert \hat{\bar{\pi}}(t') \Vert \leq \Vert \bar{\pi}(t') \Vert$. Observe that $\tau_i(t) + \zeta_j$ is upper bounded by $2T_d +T_a$ by Assumption \ref{assump:2}, we have
\begin{align*}
&\Vert\underline{\lambda}(t-1) - \bar{\lambda}(t-1) - q(t-1-\tau_i(t))\Vert \\
 =&  \Vert\sum_{t' = t - 2T_d - T_a }^{t-1}  (\hat{\underline{\pi}}(t') - \hat{\bar{\pi}}(t'))\Vert \\
\leq &\sum_{t' = t - 2T_d - T_a }^{t-1} (\Vert \hat{\underline{\pi}}(t') \Vert + \Vert \hat{\bar{\pi}}(t') \Vert ) \\
\leq &\sum_{t' = t - 2T_d - T_a }^{t-1} (\Vert \underline{\pi}(t') \Vert + \Vert \bar{\pi}(t') \Vert )
\end{align*}
which leads to the second last inequality of (\ref{eq:nabla_mu_i}).

Sum (\ref{eq:nabla_mu_i}) over $i$ and $n+i$ where $i = 1,...,n$, we have
\begin{equation*}
\Vert \nabla D(\lambda(t-1)) - \mu(t)\Vert \leq 2 \Vert X \Vert_F \sum_{t'=t-2T_d - T_a}^{t-1}\Vert\pi(t')\Vert
\end{equation*}
(ii) By some similar arguments in the derivation of the second last inequality in (\ref{eq:nabla_mu_i}), we have
\begin{align*}
\Vert v(t)-v(t-1)\Vert &\leq  \sigma_{\max}(X) \Vert q(t)-q(t-1)\Vert \\ 
                       &\leq  \sigma_{\max}(X) \sum_{t' = t-T_a-T_d}^{t}( \Vert\bar{\pi}(t')\Vert + \Vert\underline{\pi}(t')\Vert)
\end{align*}
\end{proof}
\begin{lemma} \label{lem:Dinqual}
There exists a positive constant $A$ s.t. when $\epsilon < \frac{1}{A}$, $D(\lambda(t))$ is bounded, and $$||\pi(t)\Vert \rightarrow 0 (t \rightarrow \infty). $$ 
\end{lemma}
\begin{proof}
Apply Proposition~\ref{prop:hessianD}, Lemma~\ref{lem:nablapi}(i) and Lemma~\ref{lem:Dmuerror} to the second order Taylor expansion of $D(\lambda(t))$, we have
\begin{align} \label{eq:D_inc}
D(\lambda(t))  &=   D(\lambda(t-1)) + \nabla D(\lambda(t-1))^T \pi(t) \nonumber \\ 
				& + \frac{1}{2} \pi^T(t) \nabla^2 D\pi(t) \nonumber\\
				 &= D(\lambda(t-1)) + (\nabla D(\lambda(t-1))-\mu(t))^T\pi(t) \nonumber\\
				 & + \mu(t)^T\pi(t) + \frac{1}{2} \pi^T(t) \nabla^2 D\pi(t) \nonumber\\				
				&\geq   D(\lambda(t-1)) -  \Vert \nabla D(\lambda(t-1))-\mu(t) \Vert \Vert \pi(t) \Vert \nonumber\\
				& +(\frac{1}{\epsilon} - \sigma_{max}(X))\left\Vert \pi(t) \right\Vert^2 \nonumber\\
				&\geq  D(\lambda(t-1)) - 2\Vert X \Vert_F \sum_{t'=t-2T_d - T_a}^{t-1}\Vert\pi(t')\Vert \Vert\pi(t)\Vert \nonumber\\
				& +(\frac{1}{\epsilon} - \sigma_{max}(X)) \left\Vert \pi(t) \right\Vert^2 \nonumber\\
				&\geq  D(\lambda(t-1)) + (\frac{1}{\epsilon} - \sigma_{max}(X))\left\Vert \pi(t) \right\Vert^2 \nonumber\\
				& - \Vert X \Vert_F \sum_{t' = t-2T_d - T_a}^{t-1} (\Vert \pi(t)\Vert^2 + \Vert \pi(t')\Vert^2) 
\end{align}

Sum (\ref{eq:D_inc}) from $t=1$ to $T$, we have
\begin{align} \label{eq:D_sum}
D(\lambda(T))  &\geq   D(\lambda(0)) + (\frac{1}{\epsilon} - \sigma_{max}(X))\sum_{t=1}^{T} \Vert \pi(t)\Vert^2 \nonumber \\
& \  - \Vert X \Vert_F \sum_{t=1}^{T} \sum_{t' = t-2T_d - T_a}^{t-1} (\Vert \pi(t)\Vert^2 + \Vert \pi(t')\Vert^2) \nonumber \\
&\geq D(\lambda(0)) + (\frac{1}{\epsilon} - \sigma_{max}(X))\sum_{t=1}^{T} \Vert \pi(t)\Vert^2 \nonumber\\
& \ - \Vert X \Vert_F (2T_d + T_a) \sum_{t=1}^{T} \Vert \pi(t)\Vert^2 \nonumber\\
& \ - \Vert X \Vert_F \sum_{t=1}^{T} \sum_{t' = t-2T_d - T_a}^{t-1}  \Vert \pi(t')\Vert^2 \nonumber\\
& \geq D(\lambda(0)) + (\frac{1}{\epsilon} - A)\sum_{t=1}^{T} \Vert \pi(t)\Vert^2 
\end{align}
where $A = \sigma_{max}(X) +2\Vert X\Vert_F (2T_d+T_a)$. Select $\epsilon$ small enough such that $$ \frac{1}{\epsilon} - A > 0$$
Since $D(\lambda(T))$ is upper bounded (assume the primal problem is feasible), by (\ref{eq:D_sum}), $\sum_{t=0}^{T}\Vert\pi(t)\Vert^2$ is also upper bounded. Because $\sum_{t=0}^{T}\Vert\pi(t)\Vert^2$ is a series consisting of non-negative terms, we must have $$ \Vert\pi(t)\Vert \rightarrow 0 (t \rightarrow \infty).$$
\end{proof}

We now prove Theorem \ref{thm:asyn}.
\begin{proof}
We first show there must exist one accumulation point of the sequence $\{\lambda(t)\}$. According to (\ref{eq:D_sum}), $D(\lambda(t))$ is lower bounded by $D(\lambda(0))$, so $\underline{\lambda}(t) - \bar{\lambda}(t)$ is bounded, and $\bar{\lambda}(t)^T(v^{par}-\bar{v}) + \underline{\lambda}(t)^T(\underline{v}-v^{par})$ is lower bounded. Since for each $i$ at least one of $(v_i^{par}-\bar{v}_i)$ and $(\underline{v}_i-v_i^{par})$ is negative, we must have $\underline{\lambda}_i(t)$ and $\bar{\lambda}_i(t)$ are bounded. So the set $\{\lambda\geq 0 | D(\lambda)\geq D(\lambda(0))\}$ is compact, then the existence of accumulation points follows from Weierstrass Theorem \cite{rudin1964principles}.

We next show every accumulation point $\lambda^*$ of $\{\lambda(t)\}$ maximizes the dual problem (\ref{eq:dual}). Let subsequence ${\lambda(t_k)}$ converges to $\lambda^*$ with $t_{k+1}-t_k \geq T_a, \forall k$. Define $t_{i,k} = \max \{t\leq t_k-1: t+1 \in T_i\}$. By Assumption \ref{assump:2}, $\{t_{i,k}\}_{k=1}^{\infty}$ satisfies $t_k - T_a \leq t_{i,k}< T_k$ and hence $t_{i,k}\rightarrow \infty$ as $k\rightarrow\infty$.
With Lemma~\ref{lem:Dmuerror} (i), $$\left\Vert\nabla D(\lambda(t-1)) - \mu(t)\right\Vert \leq 2\Vert X\Vert_F \sum_{t'=t-2T_d-T_a}^{t-1}\left\Vert\pi(t')\right\Vert \rightarrow 0$$
Therefore, $$\lim_k \bar{\mu}_i(t_{k}+1)=\lim_k \frac{\partial D(\lambda(t_{k}))}{ \partial\bar{\lambda}_i} = \frac{\partial D(\lambda^*)}{\partial\bar{\lambda}_i}$$ 
By Lemma~\ref{lem:nablapi}  (ii), $\Vert v(t) - v(t-1)\Vert\rightarrow 0$, hence 
$$\Vert \bar{\mu}_i(t_k+1) - \bar{\mu}_i(t_{i,k}) \Vert \leq \sum_{t' = t_k-T_d-T_a}^{t_k+1} |v_i(t')-v_i(t'-1)| \rightarrow 0$$
Therefore $\lim_k \bar{\mu}_i(t_{i,k}) = \lim_k \bar{\mu}_i(t_{k}+1) = \partial D(\lambda^*)/\partial\bar{\lambda}_i$. Also since $$|\bar{\lambda}_i(t_k) - \bar{\lambda}_i(t_{i,k})| \leq \sum_{t' = t_k - T_a}^{t_k} |\bar{\pi}_i(t')|\rightarrow 0$$
we have $\lim_k \bar{\lambda}_i(t_k) = \lim_k\bar{\lambda}_i(t_{i,k})$. So,
\begin{align*}
[\bar{\lambda}_i^*+\epsilon \frac{\partial \nabla D(\lambda^*)}{\partial \bar{\lambda}_i}]^+ - \bar{\lambda}_i^* &= \lim_k \Big[ [\bar{\lambda}_i(t_k)+\epsilon\bar{\mu}_i(t_k+1)]^+ - \bar{\lambda}_i(t_k) \Big] \\
   &= \lim_k \Big[ [\bar{\lambda}_i(t_{i,k})+\epsilon \bar{\mu}_i(t_{i,k})]^+ - \bar{\lambda}_i(t_{i,k})\Big] \\
   &= \lim_k \bar{\pi}_i(t_{i,k}+1) = 0
\end{align*}
Similarly, $[\underline{\lambda}_i^*+\epsilon \frac{\partial \nabla D(\lambda^*)}{\partial \underline{\lambda}_i}]^+ - \underline{\lambda}_i^* =0$. Apply this equality for all $i$ and we can get $[\lambda^*+\epsilon \nabla D (\lambda^*)]^+ = \lambda^*$. Apply Projection Theorem\cite{bertsekas1999nonlinear}, we have $$-\epsilon\nabla D(\lambda^*)^T(\lambda-\lambda^*)\geq 0,\  \  \forall \lambda \geq 0$$
Which establishes the optimality of $\lambda^*$ in the dual problem, and thus also the optimality of $\underline{\lambda}^*- \bar{\lambda}^*$ in the primal problem.

Next we will prove the uniqueness of the maximizers of the dual problem (\ref{eq:dual}). Let $\lambda'$ be a maximizer of $D(\lambda)$, then $\underline{\lambda}'-\bar{\lambda}'$ is the minimizer of the primal problem. Since $X$ is positive definite, the primal problem has a unique minimizer. Hence $\underline{\lambda}'-\bar{\lambda}'$ is unique. By complementary slackness in KKT condition\cite{bertsekas1999nonlinear}, $\forall i$, $\bar{\lambda}_i'([X(\underline{\lambda}'-\bar{\lambda}')]_i+v^{par}_i-\bar{v}_i)=0$ and $\underline{\lambda}_i'(\underline{v}_i-[X(\underline{\lambda}'-\bar{\lambda}')]_i-v^{par}_i)=0$. Because $([X(\underline{\lambda}'-\bar{\lambda}')]_i+v^{par}_i-\bar{v}_i)$ and $(\underline{v}_i-[X(\underline{\lambda}'-\bar{\lambda}')]_i-v^{par}_i)$ cannot be both zero, at least one of $\bar{\lambda}'_i$ and $\underline{\lambda}'_i$ has to be zero. Combining this with the fact that $\underline{\lambda}'- \bar{\lambda}'$ is unique, we have $\lambda'$ is unique, i.e. the dual problem (\ref{eq:dual}) has a unique maximizer. 

At last, since any accumulation point of $\{\lambda(t)\}$ is the maximizer of (\ref{eq:dual}), the accumulation point of $\{\lambda(t)\}$ is unique. According to \cite{rudin1964principles} this means $\lambda(t)$ will converge to the maximizer of the dual problem (\ref{eq:dual}). Also, by (\ref{eq:asyn_noupdate}), (\ref{eq:asyn_update}) and Assumption \ref{assump:2}, $q(t)$ will converge to the minimizer of the primal problem (\ref{eq:min_q}) and $v(t)$ will converge to a point $v^*$ in $[\underline{v},\bar{v}]$.



\end{proof}

\end{appendix}
\end{document}